\documentclass[11pt]{article}
\usepackage{amsthm,amsmath,amssymb,bbm}
\usepackage{geometry}

\def\Ex{{\mathbb E}}
\def\Pr{{\mathbb P}}

\def\ve{\varepsilon}

\def\op{\mathrm{op}}
\def\HS{\mathrm{HS}}
\def\bfi{\mathbf{i}}

\newtheorem{thm}{Theorem}

\newtheorem{prop}[thm]{Proposition}
\newtheorem{cor}[thm]{Corollary}

\theoremstyle{definition}
\newtheorem{exm}{Example}

\theoremstyle{remark}
\newtheorem{rem}{Remark}

\title{Moment estimates for chaoses generated by symmetric random variables with logarithmically convex tails}
\author{Konrad Kolesko\thanks{The paper was prepared while the author held a post-doctoral position
at Warsaw Center of Mathematics and Computer Science. Research supported by the NCN grant DEC-2012/05/B/ST1/00692}
\and Rafa{\l} Lata{\l}a\thanks{Research supported by the NCN grant DEC-2012/05/B/ST1/00412}}

\date{}

\begin{document}

\maketitle

\begin{abstract}
We derive two-sided estimates for random multilinear forms (random chaoses) generated by independent
symmetric random variables with logarithmically concave tails. Estimates are exact up to multiplicative constants
depending only on the order of chaos.\\[1em]%
   {\footnotesize%
     \textbf{Keywords: }Polynomial chaoses; Tail and moment estimates; Logarithmically convex tails.\par%
     \noindent\textbf{AMS MSC 2010:}  60E15 }%
\end{abstract}

\section{Introduction and Main Results.}

In this paper we study homogeneous chaoses of order $d$, i.e.\ random variables of the form
\begin{equation}
\label{homchaos}
S=\sum_{i_1,\ldots,i_d=1}^n a_{i_1,\ldots,i_d}X_{i_1} \cdots X_{i_d},
\end{equation}
where $X_1,\ldots,X_n$ are independent random variables and $(a_{i_1,\ldots,i_d})$ is a multiindexed symmetric array of
real numbers such that $a_{i_1,\ldots,i_d}=0$ whenever $i_k=i_l$ for some $k\neq l$.

Chaoses of order $d=1$ are just sums of independent r.v's. There are numerous classical results 
providing bounds for moments and tails of $S$ in this case (such as Khintchine, Rosenthal, Bernstein, Hoeffding, Prokhorov, Bennett 
inequalities, to name a few). However the situation is more delicate if one looks for two-sided estimates. 
In \cite{La_mom} two-sided bounds for $L_p$-norms $\|S\|_p=(\Ex|S|^p)^{1/p}$ were found in quite general situation.
Namely, for any $p\geq 2$ and mean zero r.v's $X_i$ a deterministic function $f_p(a_1,\ldots,a_n)$ was constructed 
such that 
\[
\frac{1}{C}f_p(a_1,\ldots,a_n)\leq \left\|\sum_{i=1}^n a_iX_i\right\|_p\leq Cf_p(a_1,\ldots,a_n),
\]  
where $C$ is a universal constant. Strictly related question concerning two-sided bounds for tails of $|S|$ was 
treated in \cite{HMS}.

The case $d\geq 2$ is much less understood.
In \cite{La_chaos} two-sided estimates for moments of Gaussian chaoses were found. In \cite{AL} moment bounds were
established in the case when $d\leq 3$ and $X_i$ are symmetric with logarithmically concave tails and for
chaoses of arbitrary order generated by symmetric exponential r.v's. The main purpose
of this note is to study the case of symmetric variables with logarithmically convex tails, i.e.\ r.v's
such that functions $t\mapsto -\ln\Pr(|X_i|\geq t)$ are convex on $[0,\infty)$. The class of variables with logconvex
tails includes in particular variables with exponential and heavy-tailed Weibull distributions.

Observe that symmetric exponential random variables are in both classes (logarithmically concave and logarithmically 
convex tails) and methods described in this note allow to avoid many technical calculations  presented in  
\cite{AL} for this particular case (cf. Theorem \ref{thm:main} and Remark \ref{rem:3} with Theorem 3.4 in \cite{AL}). 

One of the main tools in the study of random chaoses is the decoupling technique (cf. the monograph \cite{dlPG} for its
various applications). It states that the asymptotic behaviour of the homogeneous chaos 
$S$ is the same as of its decoupled counterpart $\tilde{S}$ defined by the formula
\[
\tilde{S}:=\sum_{i_1,\ldots,i_d=1}^n a_{i_1,\ldots,i_d}X_{i_1}^1 \cdots X_{i_d}^d,
\]
where $(X_i^j)_{1\leq i\leq n}$, $j=1,\ldots,d$ are independent copies of the sequence $(X_i)_{i\leq n}$.
In particular the result of Kwapie\'n \cite{K} (see also \cite{dlPMS} for its more general version)
says that for any symmetric multiindexed matrix $(a_{i_1,\ldots,i_d})$ such that $a_{i_1,\ldots,i_d}=0$ 
whenever $i_k=i_l$ for some  $k\neq l$ and any $p\geq 1$,
\begin{equation}
\label{eq:decmom}
\frac{1}{C(d)}\|\tilde{S}\|_p\leq \|S\|_p\leq C(d)\|\tilde{S}\|_p,
\end{equation}
where $C(d)$ is a positive constant, which depends only on $d$. 

\smallskip

Before we state main results we need to introduce some notation. By $C$ (resp. $C(d)$) we denote positive universal constants 
(positive constants depending only on the parameter $d$).
In all cases values of constants may differ at each occurrence.
To simplify the notation we write $A\sim B$ ($A\sim_d B$) if $\frac{1}{C}A\leq B\leq CA$ ( $\frac{1}{C(d)}A\leq B\leq C(d)A$
resp.).
 
The $p$th moment of a random variable $X$ is $\|X\|_p^p=\Ex|X|^p$.

For $\bfi\in \{1,\ldots,n\}^d$ and $I\subset [d]:=\{1,\ldots,d\}$ we write $\bfi_I:=(i_k)_{k\in I}$. 
For $I\subset [d]$ by ${\cal P}(I)$ we denote the family of all partitions of $I$ into pairwise disjoint 
subsets. If ${\cal J}=\{I_1,\ldots,I_k\}\in {\cal P}(I)$ and $(a_{\bfi})$ is a multiindexed matrix we set
\[
\|(a_{\bfi})\|_{{\cal J}}=\|(a_{\bfi})_{{\bfi}_I}\|_{{\cal J}}
:=\sup\left\{\sum_{{\bfi_I}}a_{\bfi}\prod_{l=1}^kx_{\bfi_{I_l}}\colon 
\sum_{\bfi_{I_l}}x_{\bfi_{I_l}}^2\leq 1, 1\leq l\leq k \right\}.
\]

\smallskip

Now we are ready to state our main theorem.

\begin{thm}
\label{thm:main}
Let $(X_i^j)_{i\leq n,j\leq d}$ be independent symmetric r.v's with logarithmically convex tails such that 
$\Ex|X_i^j|^2=1$ for all $i,j$. 
For any multiindexed matrix $(a_{\bfi})$ and any $p\geq 2$ we have
\begin{align*}
\left\|\sum_{\bfi} a_{\bfi}X_{i_1}^1 \cdots X_{i_d}^d\right\|_p
\sim_d \left(\sum_{I\subset [d]}\sum_{\bfi_I}\sum_{{\cal J}\in {\cal P}(I^c)}p^{p|{\cal J}|/2}
\|(a_{\bfi})_{\bfi_{I^c}}\|_{{\cal J}}^p
\prod_{j\in I}\|X_{i_j}^j\|_p^p\right)^{1/p}
\\
\sim_d \sum_{I\subset [d]}\sum_{{\cal J}\in {\cal P}(I^c)}p^{|{\cal J}|/2}
\left(\sum_{\bfi_I}\|(a_{\bfi})_{\bfi_{I^c}}\|_{{\cal J}}^p \prod_{j\in I}\|X_{i_j}^j\|_p^p\right)^{1/p}.
\end{align*}
\end{thm}

\begin{rem}
Theorem \ref{thm:main} (and \eqref{eq:decmom}) may be used to obtain two-sided estimates for 
moments of nonhomogeneous chaoses of the form
\[
S=\sum_{j=0}^d\sum_{i_1,\ldots,i_j=1}^n a_{i_1,\ldots,i_j}^jX_{i_1} \cdots X_{i_j},
\]
where $a_{\emptyset}^0$ is a number and $(a_{i_1,\ldots,i_j}^j)_{i_1,\ldots,i_j}$, $j=1,\ldots,d$ are 
multiindexed matrices such that
$a_{i_1,\ldots,i_j}^j=0$ if $i_k=i_l$ for some $1\leq k<l\leq j$.
Indeed, in this case we have 
(cf. \cite[Lemma 2]{K} or \cite[Proposition 1.2]{AL})
\[
\frac{1}{C(d)}\sum_{j=0}^d\left\|\sum_{i_1,\ldots,i_j=1}^n a_{i_1,\ldots,i_j}^j X_{i_1} \cdots X_{i_j}\right\|_p
\leq \|S\|_p\leq \sum_{j=0}^d\left\|\sum_{i_1,\ldots,i_j=1}^n a_{i_1,\ldots,i_j}^j X_{i_1} \cdots X_{i_j}\right\|_p.
\]
\end{rem}

\begin{rem}
Theorem \ref{thm:main} yields also tail bounds for chaoses based on symmetric r.v's with logarithmically convex
tails. Obviously we have $\Pr(|S|\geq e\|S\|_p)\leq e^{-p}$. Moreover, if $\|S\|_{2p}\leq \lambda \|S\|_p$ 
for all $p\geq 2$
then by the Paley-Zygmund inequality it is not hard to show that $\Pr(|S|\geq \|S\|_p/C(\lambda))\geq e^{-p}$
for $p\geq p(\lambda)$. Observe also that if $\|X_i\|_{2p}\leq \mu\|X_i\|_p$ for all $i$ then, 
by Theorem \ref{thm:main}, $\|S\|_{2p}\leq C(d)\mu^d\|S\|_p$. 
\end{rem}

\begin{rem}
\label{rem:3}
Suppose that functions $p\mapsto \|X_i^j\|_p$ grow at most exponentially, i.e\ there exist constants
$\alpha,\beta$ such that
\[
\|X_i^j\|_{p}\leq \alpha e^{\beta p} \quad\mbox{ for all }i\leq n,\ j\leq d \mbox{ and }p\geq 2.
\]
Then for $p\geq 2$,
\begin{align*}
\sum_{I\subset [d]}\sum_{{\cal J}\in {\cal P}(I^c)}p^{|{\cal J}|/2}
&\left(\sum_{\bfi_I}\|(a_{\bfi})_{\bfi_{I^c}}\|_{{\cal J}}^p \prod_{j\in I}\|X_{i_j}^j\|_p^p\right)^{1/p}
\\
&~\sim_{d,\alpha,\beta}
\sum_{I\subset [d]}\sum_{{\cal J}\in {\cal P}(I^c)}p^{|{\cal J}|/2}
\max_{\bfi_I}\|(a_{\bfi})_{\bfi_{I^c}}\|_{{\cal J}} \prod_{j\in I}\|X_{i_j}^j\|_p.
\end{align*}

To see this, observe first that for any $\gamma\geq 0$ and $p\geq 3$, 
\begin{align*}
\|x\|_{\ell^p}
&\leq \|x\|_{\ell^2}^{2/ p}\|x\|_{\ell^{\infty}}^{1-2/p}
= \left\|e^{-\gamma p}x\right\|_{\ell^2}^{2/ p}\left\|e^{\frac {2\gamma p}{p-2}}x\right\|_{\ell^{\infty}}^{1-2/p}
\\
&\leq \frac{2}{p}\left\|e^{-\gamma p}x\right\|_{\ell^2}+
\frac{p-2}{p}\left\|e^{\frac {2\gamma p}{p-2}}x\right\|_{\ell^{\infty}}
\leq 
e^{-\gamma p}\|x\|_{\ell^2}+e^{6\gamma}\|x\|_{\ell^{\infty}}.
\end{align*}
Therefore 
\begin{equation}
\label{eq:interpol}
\|x\|_{\ell^p}\leq C(\gamma)\left(e^{-\gamma p}\|x\|_{\ell^2} +\|x\|_{\ell^{\infty}}\right)
\quad \mbox{for }p\geq 2. 
\end{equation}

Fix $I\subset [d]$ and ${\cal J}\in {\cal P}(I^c)$. We have $p^{|{\cal J}|/2}e^{-p}\leq p^{d/2}e^{-p}\leq C(d)$ and
\[
\sum_{\bfi_I}\left(\|(a_{\bfi})_{\bfi_{I^c}}\|_{{\cal J}}\prod_{j\in I}\|X_{i_j}^j\|_p\right)^2\leq 
\sum_{\bfi_I}\alpha^{2|I|}e^{2\beta p|I|}\sum_{\bfi_{I^c}}a_{\bfi}^2=
\alpha^{2|I|}e^{2\beta p|I|}\|(a_{\bfi})\|_{\{[d]\}}^2.
\]
Hence \eqref{eq:interpol}, applied with $\gamma=d\beta+1$, yields 
\begin{align*}
p^{|{\cal J}|/2}\left(\sum_{\bfi_I}\|(a_{\bfi})_{\bfi_{I^c}}\|_{{\cal J}}^p \prod_{j\in I}\|X_{i_j}\|_p^p\right)^{1/p}
&
\\
\leq C(d,\alpha,\beta)
&\left(\|(a_{\bfi})\|_{\{[d]\}}+p^{|{\cal J}|/2} 
\max_{\bfi_I}\|(a_{\bfi})_{\bfi_{I^c}}\|_{{\cal J}}\prod_{j\in I}\|X_{i_j}\|_p\right).
\end{align*}
\end{rem}

\begin{exm}
Observe that $\|(a_{ij})\|_{\{1,2\}}$ is the Hilbert-Schmidt and $\|(a_{ij})\|_{\{1\},\{2\}}$
the operator norm of a matrix $(a_{ij})$. Thus \eqref{eq:decmom} and Theorem \ref{thm:main}  yield
the following two-sided estimate for chaoses of order two ($(a_{ij})$ is a symmetric matrix with zero diagonal),
\begin{align*}
\left\|\sum_{i,j=1}^na_{i,j}X_iX_j\right\|_p
&\sim \left\|\sum_{i,j=1}^na_{i,j}X_i^1X_j^2\right\|_p
\\
&\sim
p\|(a_{i,j})\|_{\op}+p^{1/2}\|(a_{i,j})\|_{\HS}
+p^{1/2}\left(\sum_{i}\|X_i\|^p\left(\sum_j a_{ij}^2\right)^{p/2}\right)^{1/p}
\\
&\phantom{\sim}
+\left(\sum_{i,j}|a_{i,j}|^p\|X_i\|_p^p\|X_j\|_p^p\right)^{1/p}.
\end{align*}
\end{exm}

\begin{exm}
If $X_i$ have symmetric exponential distribution with the density $\frac{1}{2}e^{-|x|}$ then
$\|X_i\|_p=\Gamma(p+1)^{1/p}\sim p$ and we obtain by Theorem \ref{thm:main} and Remark \ref{rem:3}
\begin{align*}
\left\|\sum_{\bfi} a_{\bfi}X_{i_1}^1 \cdots X_{i_d}^d\right\|_p
&\sim_d \sum_{I\subset [d]}\sum_{{\cal J}\in {\cal P}(I^c)}p^{|I|+|{\cal J}|/2}
\left(\sum_{\bfi_I}\|(a_{\bfi})_{\bfi_{I^c}}\|_{{\cal J}}^p\right)^{1/p}
\\
&\sim_d \sum_{I\subset [d]}\sum_{{\cal J}\in {\cal P}(I^c)}p^{|I|+|{\cal J}|/2}
\max_{\bfi_I}\|(a_{\bfi})_{\bfi_{I^c}}\|_{{\cal J}}.
\end{align*}
\end{exm}

\begin{exm}
If $X_i$ have symmetric Weibull distribution with scale parameter $1$ and shape parameter $r\in (0,1]$,
i.e. $\Pr(|X_i|\geq t)=\exp(-t^r)$ for $t\geq 0$ then $\|X_i\|_p=\Gamma(p/r+1)^{1/p}$ and
\[
\left\|\sum_{\bfi} a_{\bfi}X_{i_1}^1 \cdots X_{i_d}^d\right\|_p
\sim_d \sum_{I\subset [d]}\sum_{{\cal J}\in {\cal P}(I^c)}\Gamma\left(\frac{p}{r}+1\right)^{|I|}p^{|{\cal J}|/2}
\left(\sum_{\bfi_I}\|(a_{\bfi})_{\bfi_{I^c}}\|_{{\cal J}}^p\right)^{1/p}.
\]
However by Stirling's formula $\Gamma(p/r+1)^{1/p}\sim_r p^{1/r}$, hence 
\begin{align*}
\left\|\sum_{\bfi} a_{\bfi}X_{i_1}^1 \cdots X_{i_d}^d\right\|_p
&\sim_{d,r} \sum_{I\subset [d]}\sum_{{\cal J}\in {\cal P}(I^c)}p^{|I|/r+|{\cal J}|/2}
\left(\sum_{\bfi_I}\|(a_{\bfi})_{\bfi_{I^c}}\|_{{\cal J}}^p\right)^{1/p}
\\
&\sim_{d,r} \sum_{I\subset [d]}\sum_{{\cal J}\in {\cal P}(I^c)}p^{|I|/r+|{\cal J}|/2}
\max_{\bfi_I}\|(a_{\bfi})_{\bfi_{I^c}}\|_{{\cal J}},
\end{align*}
where the last estimate follows from Remark \ref{rem:3}.
\end{exm}

Theorem \ref{thm:main} may be used to derive upper moment and tail bounds for chaoses based on variables 
whose moments are dominated by moments of variables with logconvex tails. Here is a sample
result in such direction.

\begin{cor}
\label{cor:dombyWeibull}
Let $r\in (0,1]$, $A<\infty$ and suppose that $(X_i^j)_{i\leq n,j\leq d}$ are independent centred r.v's such that 
$\|X_i^j\|_p\leq Ap^{1/r}$ for all $i,j$ and $p\geq 2$. 
For any multiindexed matrix $(a_{\bfi})$ and any $p\geq 2$ we have
\begin{equation}
\label{eq:momdombyWeibull}
\left\|\sum_{\bfi} a_{\bfi}X_{i_1}^1 \cdots X_{i_d}^d\right\|_p 
\leq C(r,d)A^d
\sum_{I\subset [d]}\sum_{{\cal J}\in {\cal P}(I^c)}p^{|I|/r+|{\cal J}|/2}
\max_{\bfi_I}\|(a_{\bfi})_{\bfi_{I^c}}\|_{{\cal J}}.
\end{equation}
Moreover for $t>0$,
\begin{align}
\notag
\Pr&\left(\left|\sum_{\bfi} a_{\bfi}X_{i_1}^1 \cdots X_{i_d}^d\right|\geq t \right)
\\
\label{eq:taildombyWeibull}
&\phantom{aaaaaaaaaa}\leq
2\exp\left(-\min_{I\subset [d]}\min_{{\cal J}\in {\cal P}(I^c)}
\left(\frac{t}{C'(r,d)A^d\max_{\bfi_I}\|(a_{\bfi})_{\bfi_{I^c}}\|_{{\cal J}}}\right)^{\frac{2r}{2|I|+r|{\cal J}|}}\right).
\end{align}
\end{cor}

\section{Proofs}

The following result was established by Hitczenko, Montgomery-Smith and Oleszkiewicz \cite{HMSO} 
(see \cite[Example 3.3]{La_mom} for a simpler proof).

\begin{thm}
\label{thm_HMSO}
Let $X_i$ be independent symmetric r.v's with logarithmically convex tails. Then for any $p\geq 2$,
\[
\left\|\sum_{i=1}^nX_i\right\|_p\sim \left(\sum_{i=1}^n\Ex |X_i|^p\right)^{1/p}
+\sqrt{p}\left(\sum_{i=1}^n\Ex X_i^{2}\right)^{1/2}.
\]
\end{thm}

Since $\|g\|_p\sim \sqrt{p}$ for $p\geq 2$ and standard normal ${\cal N}(0,1)$ r.v. $g$ we immediately get
the following corollary.

\begin{cor}
\label{cor_HMSO}
Let $X_i$ be independent symmetric r.v's with logarithmically convex tails and variance one and let $g_i$ be i.i.d.\
standard normal ${\cal N}(0,1)$ r.v's. Then for any scalars $a_i$,
\[
\left\|\sum_{i=1}^na_iX_i\right\|_p\sim \left(\sum_{i=1}^n|a_i|^p\Ex |X_i|^p\right)^{1/p}+
\left\|\sum_{i=1}^na_ig_i\right\|_p.
\]
\end{cor}

\begin{rem}
\label{rem:0}
Both Theorem \ref{thm_HMSO} and Corollary \ref{cor_HMSO} may be viewed as some versions of the Rosenthal inequality.
Note however that, contrary to Rosenthal's bound, provided estimates are sharp up to universal constants
that do not depend on $p$.  
\end{rem}

\begin{prop}
\label{prop:red_to_gauss}
Let $(X_i^j)_{i\leq n,j\leq d}$ be as in Theorem \ref{thm:main} and $(g_i^j)_{i\leq n,j\leq d}$ be i.i.d.\
standard normal ${\cal N}(0,1)$ r.v's. For any multiindexed matrix $(a_{\bfi})$ and 
any $p\geq 2$ we have
\begin{equation}
\label{eq:red_to_gauss}
\left\|\sum_{\bfi} a_{\bfi}X_{i_1}^1 \cdots X_{i_d}^d\right\|_p
\sim_d \left(\sum_{I\subset [d]}\sum_{\bfi_I}\left\|\sum_{\bfi_{I^c}}a_{\bfi}\prod_{j\in I^c}g_{i_j}^j\right\|_p^p
\prod_{j\in I}\|X_{i_j}^j\|_p^p\right)^{1/p}.
\end{equation}
\end{prop}

\begin{proof}
We proceed by induction with respect to $d$. For $d=1$ the estimate follows by Corollary \ref{cor_HMSO}.

To show the induction step assume that $d>1$ and the bound holds for $d-1$. By Corollary \ref{cor_HMSO} applied
conditionally we get
\begin{align}
\notag
\Big\|\sum_{\bfi} a_{\bfi}&X_{i_1}^1 \cdots X_{i_d}^d\Big\|_p
\\
\label{ind1}
&\sim
\left(\left\|\sum_{\bfi} a_{\bfi}X_{i_1}^1 \cdots X_{i_{d-1}}^{d-1}g_{i_d}^d\right\|_p^p
+\sum_{i_d}\left\|\sum_{\bfi_{[d-1]}} a_{\bfi}X_{i_1}^1 \cdots X_{i_{d-1}}^{d-1}\right\|_p^p\|X_{i_d}^d\|_p^p
\right)^{1/p}.
\end{align}
The conditional application of the induction assumption yields
\begin{equation}
\label{ind2}
\left\|\sum_{\bfi} a_{\bfi}X_{i_1}^1 \cdots X_{i_{d-1}}^{d-1}g_{i_d}^d\right\|_p
\sim_d \left(\sum_{I\subset [d-1]}\sum_{\bfi_I}\left\|\sum_{\bfi_{I^c}}a_{\bfi}\prod_{j\in I^c}g_{i_j}^j\right\|_p^p
\prod_{j\in I}\|X_{i_j}^j\|_p^p\right)^{1/p}
\end{equation}
and for any $i_d$,
\begin{equation}
\label{ind3}
\left\|\sum_{\bfi_{[d-1]}} a_{\bfi}X_{i_1}^1 \cdots X_{i_{d-1}}^{d-1}\right\|_p
\sim_d 
\left(\sum_{I\subset [d-1]}\sum_{\bfi_I}\left\|\sum_{\bfi_{[d-1]\setminus I}}a_{\bfi}\prod_{j\in [d-1]\setminus I}g_{i_j}^j\right\|_p^p
\prod_{j\in I}\|X_{i_j}^j\|_p^p\right)^{1/p}.
\end{equation}
Estimates \eqref{ind1}--\eqref{ind3} imply \eqref{eq:red_to_gauss}.
\end{proof}

Theorem \ref{thm:main} immediately follows by Proposition \ref{prop:red_to_gauss} and the following
two-sided bound for moments of Gaussian chaoses \cite{La_chaos}.

\begin{thm}
\label{thm:gauss}
For any $d$ and $p\geq 2$ we have
\[
\left\|\sum_{\bfi} a_{\bfi}g_{i_1}^1 \cdots g_{i_d}^d\right\|_p
\sim_d \sum_{{\cal J}\in {\cal P}([d])}p^{|{\cal J}|/2}\|(a_{\bfi})\|_{{\cal J}}.
\]
\end{thm}

\begin{proof}[Proof of Corollary \ref{cor:dombyWeibull}]
Obviously it is enough to show \eqref{eq:momdombyWeibull} for $p=2l$, $l=1,2,\ldots$.
Let $Y_i^j$, $i\geq 1,j=1,\ldots,d$
be i.i.d.\ symmetric Weibull r.v's with scale parameter $1$ and shape parameter $r$ and 
let $(\ve_{i})$ be i.i.d.\ symmetric
$\pm 1$ r.v's, independent of the sequence $(X_i^j)$. We have (see Example 3) 
$\|X_i^j\|_{p}=\|\ve_iX_i^j\|_{p}\leq Ap^{1/r}\leq C(r)A\|Y_i^j\|_p$ for all $p\geq 2$. Thus, for any
positive integer $l$ and any scalars $(b_i)$,
\[
\left\|\sum_{i}b_iX_i^j\right\|_{2l}\leq 2\left\|\sum_{i}b_i\ve_iX_i^j\right\|_{2l}
\leq 2C(r)A\left\|\sum_{i}b_iY_i^j\right\|_{2l},
\]
where the first inequality follows by the standard symmetrization argument (cf.\ \cite[Lemma 1.2.6]{dlPG}).
Easy induction shows that
\[
\left\|\sum_{\bfi} a_{\bfi}X_{i_1}^1 \cdots X_{i_d}^d\right\|_{2l}
\leq (2C(r)A)^d \left\|\sum_{\bfi} a_{\bfi}Y_{i_1}^1 \cdots Y_{i_d}^d\right\|_{2l}
\]
and \eqref{eq:momdombyWeibull} (for $p=2l$) follows by Example 3. 

The tail bound \eqref{eq:taildombyWeibull} follows by \eqref{eq:momdombyWeibull} and Chebyshev's inequality.
\end{proof}

\noindent
{\sc Konrad Kolesko}\\
Instytut Matematyczny\\
Uniwersytet Wroc{\l}awski\\
Pl. Grunwaldzki 2/4\\
50-384 Wroc{\l}aw, Poland\\
\texttt{kolesko@math.uni.wroc.pl}

\medskip
\noindent
{\sc Rafa{\l} Lata{\l}a}\\
Institute of Mathematics\\
University of Warsaw\\
Banacha 2\\
02-097 Warszawa, Poland\\
\texttt{rlatala@mimuw.edu.pl}


\begin{thebibliography}{1}
\bibitem{AL} R.~Adamczak and R.~Lata{\l}a,
\textit{Tail and moment estimates for chaoses generated by symmetric random variables with logarithmically concave tails},
Ann. Inst. Henri Poincar\'e Probab. Stat. \textbf{48} (2012), 1103--1136.


\bibitem{dlPG} V.~H.~de la Pe\~na and E.~Gin\'e,
\textit{Decoupling: From Dependence to Independence},
Springer, New York, 1999.

\bibitem{dlPMS} V.~H.~de la Pe\~na and S.~J.~Montgomery-Smith,
\textit{Decoupling inequalities for the tail probabilities of multivariate $U$-statistics},
Ann. Probab. \textbf{23} (1995) 806--816.

\bibitem{HMSO} P.~Hitczenko, S.~J.~Montgomery-Smith and K.~Oleszkiewicz,
\textit{Moment inequalities for sums of certain independent symmetric random variables}, 
Studia Math. \textbf{123} (1997), 15--42.

\bibitem{HMS} P.~Hitczenko and S.~J.~Montgomery-Smith,
\textit{Measuring the magnitude of sums of independent random variables},
Ann. Probab. \textbf{29} (2001), 447--466.

\bibitem{K} S.~Kwapie\'n, 
\textit{Decoupling inequalities for polynomial chaos}, 
Ann. Probab. \textbf{15} (1987), 1062--1071.

\bibitem{La_mom} R.~Lata{\l}a,
\textit{Estimation of moments of sums of independent real random variables},   
Ann. Probab. \textbf{25} (1997), 1502--1513.


\bibitem{La_chaos} R.~Lata{\l}a, 
\textit{Estimates of moments and tails of Gaussian chaoses},
Ann. Probab. \textbf{34} (2006), 2315--2331.


\end{thebibliography}
\end{document}